\documentclass[11pt]{amsart}

\usepackage{amsmath}
\usepackage{amssymb}

\usepackage{amsfonts}
\usepackage{amsthm}
\usepackage{enumerate}
\usepackage[mathscr]{eucal}

\usepackage{tikz}

\usepackage{bbm}
\usepackage{graphicx}
\usepackage{verbatim}

\newcommand{\Be}{\begin{equation}}
\newcommand{\Ee}{\end{equation}}

\newcommand{\Bea}{\begin{eqnarray}}
\newcommand{\Eea}{\end{eqnarray}}
\newcommand{\Beas}{\begin{eqnarray*}}
\newcommand{\Eeas}{\end{eqnarray*}}
\newcommand{\Benu}{\begin{enumerate}}
\newcommand{\Eenu}{\end{enumerate}}
\newcommand{\Bi}{\begin{itemize}}
\newcommand{\Ei}{\end{itemize}}

\def\intslash{\rlap{\kern  .32em $\mspace {.5mu}\backslash$ }\int}
\def\qsl{{\rlap{\kern  .32em $\mspace {.5mu}\backslash$ }\int_{Q_x}}}

\def\rr{\mathbb R}

\def\emph#1{{\it #1 }}

\def\HF{{\text{\rm HF}}}

\def\supp{{\text{\rm supp}}}

\def\inn#1#2{\langle#1,#2\rangle}

\def\card{\text{\rm card}}
\def\lc{\lesssim}
\def\gc{\gtrsim}

\def\eps{\varepsilon}

              \def\Om{\Omega}

\def\fA{{\mathfrak {A}}}

\def\fL{{\mathfrak {L}}}

\def\fS{{\mathfrak {S}}}

\def\bbN{{\mathbb {N}}}

\def\bbZ{{\mathbb {Z}}}

\def\sH{{\mathscr {H}}}

\def\cP{{\mathcal {P}}}

\def\cS{{\mathcal {S}}}

\def\cV{{\mathcal {V}}}

 at 10 true pt

\def\be#1{\begin{equation}\label{ #1}}
\def\endeq{\end{equation}}
\def\endal{\end{align}}
\def\bas{\begin{align*}}
\def\eas{\end{align*}}
\def\bi{\begin{itemize}}
\def\ei{\end{itemize}}

\def\eps{\varepsilon}

\def\emph#1{{\it #1}}
\def\textbf#1{{\bf #1}}

\usepackage{bbm}
\def\bbone{{\mathbbm 1}}

\def\Gf{G_{k,l}^{j,N}}

\theoremstyle{plain}
  \newtheorem{theorem}{Theorem}[section]

   \newtheorem{proposition}[theorem]{Proposition}
   
   \newtheorem{lemma}[theorem]{Lemma}

\newtheorem*{thm}{Theorem}

\theoremstyle{remark}

\theoremstyle{definition}
   \newtheorem{definition}[theorem]{Definition}

\begin{document}

\title[Lower bounds for Haar projections] {Lower bounds for 
Haar projections: Deterministic examples}

\author{Andreas Seeger \ \ \ \ \ \ \ \ \ Tino Ullrich}

\address{Andreas Seeger \\ Department of Mathematics \\ University of Wisconsin \\480 Lincoln Drive\\ Madison, WI,
53706, USA} \email{seeger@math.wisc.edu}

\address{Tino Ullrich\\
Hausdorff Center for Mathematics\\ Endenicher Allee 62\\
53115 Bonn, Germany} \email{tino.ullrich@hcm.uni-bonn.de}
\begin{abstract} In a previous paper by the authors the existence of Haar projections with growing
norms in Sobolev-Triebel-Lizorkin spaces has been shown via a probabilistic argument. This existence was sufficient
to determine the precise range of Triebel-Lizorkin spaces for which the Haar system is an unconditional
basis. The aim of the present paper is to give simple deterministic 
examples of Haar projections that show this growth behavior in the respective range of parameters. 
\end{abstract}

\date\today
\thanks{Research supported in part  by the National Science Foundation
and the DFG Emmy-Noether Programme UL403/1-1}

\maketitle



\section{Introduction} In the recent paper \cite{su} the authors considered the question in which range of parameters
the Haar system is an unconditional basis in the Triebel-Lizorkin space $F^s_{p,q}(\rr)$, $1<p,q<\infty$. It turned out
that this is the case if and only if
\Be\label{f1}
    \max\{-1/p',-1/q'\}<s<\min\{1/p,1/q\}\,.
\Ee
\begin{figure}\label{triebelfigure}
 \begin{center}
\begin{tikzpicture}[scale=2]
\draw[->] (-0.1,0.0) -- (1.1,0.0) node[right] {$\frac{1}{p}$};
\draw[->] (0.0,-1.1) -- (0.0,1.1) node[above] {$s$};

\draw (1.0,0.03) -- (1.0,-0.03) node [below] {$1$};
\draw (0.03,1.0) -- (-0.03,1.00) node [left] {$1$};
\draw (0.03,.5) -- (-0.03,.5) node [left] {$\frac{1}{2}$};
\draw (0.03,-.5) -- (-0.03,-.5) node [left] {$-\frac{1}{2}$};
\draw (0.03,-1.0) -- (-0.03,-1.00) node [left] {$-1$};

\draw[fill=black!70, opacity=0.5] (0.0,-.5) -- (0.0,0.0) -- (.5,.5) -- (1.0,0.5) -- (1.0,0.0) -- (.5,-.5) -- (0.0,-.5);
\draw (0.0,-1.0) -- (0.0,0.0) -- (1.0,1.0) -- (1.0,0.0) -- (0.0,-1.0);
\draw (0.85,0.65) node {};
\draw (0.15,-0.65) node {};
\end{tikzpicture}
\caption{Domain for an unconditional basis in spaces $L^s_p$}\label{fig1}
\end{center}
\end{figure}
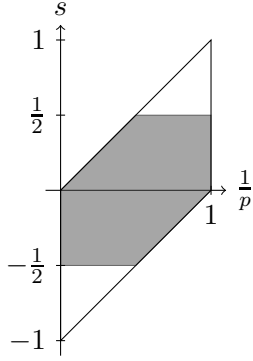
The Haar functions \eqref{Haar_f} belong to the spaces 
$F_{p,q}^s$ and $B^s_{p,q}$ if $-1/p'<s<1/p$. Moreover, by results in \cite{triebel73},  \cite{triebel78},
\cite{triebel-bases} they form  an unconditional basis in $B^s_{p,q}$ in that range. More recently, it was
shown by Triebel in \cite{triebel-bases} that in the more restrictive range \eqref{f1} the Haar system is an
unconditional basis also on $F^s_{p,q}$, and, as a special case when $q=2$, in the $L^p$ Sobolev space $L^s_p$. Triebel 
\cite{triebelproblem} asked what happens for the remaining cases corresponding to the upper and
lower  triangles in Figure \ref{triebelfigure}.


In \cite{su}  the necessity of the condition \eqref{f1} was established by showing the existence of subsets  $E$ of
the Haar system $\sH$, see \eqref{HaarS} below, for which the corresponding projections
$$
    P_{E}f = \sum\limits_{h_{j,k} \in E} 2^j\langle f, h_{j,k}\rangle h_{j,k}
$$
are not uniformly bounded in the spaces $F^s_{p,q}$ if $1<p<q$, $1/q\leq s\leq 1/p$ and $1<q<p<\infty$, $-1/p'\leq
s\leq -1/q'$. This  shows the failure of unconditional convergence of Haar expansions in the respective spaces. The
proof of the existence of such projections and sharp growth rates of their norms was based on a probabilistic argument. 

The purpose of the present paper is to present a constructive, non-pro\-ba\-bilistic argument. It turns out 
that the  families of projections providing sharp growth rates in terms of the Haar frequency set $\HF(E)$ can be
easily written down; however the proof of the lower bounds, with concrete testing functions 
  is rather technical.

We consider the Haar system on the real line given by
\Be\label{HaarS}
  \sH = \{h_{j,\mu}~:~\mu\in \bbZ, j=-1,0,1,2,...\}\,,
\Ee
where for $j \in \bbN\cup \{0\}$, $\mu\in \bbZ$, the function $h_{j,\mu}$ is
defined by 
\Be\label{Haar_f}
  h_{j,\mu}(x)=\bbone_{I^+_{j,\mu}}(x)-\bbone_{I^-_{j,\mu}}(x)\,,
\Ee
and $h_{-1,\mu}$ is the characteristic function of the interval
$[\mu,\mu+1)$. The intervals
$I^{+}_{j,\mu}=[2^{-j}\mu, 2^{-j}(\mu+1/2))$ and
$I^{-}_{j,\mu}=[2^{-j}(\mu+1/2), 2^{-j}(\mu+1))$ represent the dyadic children
of the usual dyadic interval $I_{j,\mu}=[2^{-j}\mu, 2^{-j}(\mu+1))$. To formulate the main result in \cite{su} we  say
that the Haar frequency of $h_{j,\mu}$ is $2^j$. For a set $E$ of Haar functions we define the Haar frequency set
$\HF(E)$ as the set of all $2^j$ for which 
$j\in \bbN\cup\{0\}$ and  $2^j$ is the Haar frequency of some $h\in E$.

\begin{thm} \cite{su}. (i)
 Let $1<p<q<\infty$  and $1/q<s<1/p$.
Given any set $A\subset \{2^k: k\ge 0\}$  of cardinality $\ge 2^N$ there is 
a subset $E$ of $\sH$  consisting of Haar functions supported in $[0,1]$ such that
$\HF(E)\subset A$ and such that
$$\|P_E\|_{F^s_{p,q}\to F^s_{p,q}} \ge c(p,q,s) 2^{N(s-1/q)}\,.$$

(ii)  Let $1<q<p<\infty$  and $-1+1/p<s<-1+1/q$.
Given any set $A\subset \{2^k: k\ge 0\}$  of cardinality $\ge 2^N$ there is 
a subset $E$ of $\sH$  consisting of Haar functions supported in $[0,1]$ such that
$\HF(E)\subset A$ and 
$$\|P_E\|_{F^s_{p,q}\to F^s_{p,q}} \ge c(p,q,s) 2^{N(\frac 1q -s-1)}\,.$$
\end{thm}
There are also  lower bounds in terms of powers of $N$ for the endpoint cases $F^{-1/q'}_{p,q}$, $p>q$ and
$F_{p,q}^{1/q}$, $p<q$.
We note that the result of the theorem is sharp
since there  are the  corresponding upper bounds (\cite{su})
\[
\|P_E\|_{F^s_{p,q}\to F^s_{p,q}} \le C(p,q,s) \big(\#(\HF(E))\big)^{s-\frac 1q},
\]
for $1<p<q<\infty$, $1/q<s<1/p$, and
\[
\|P_E\|_{F^s_{p,q}\to F^s_{p,q}} \le C(p,q,s) \big(\#(\HF(E))\big)^{\frac 1q -s-1},
\]
for $1<q<p<\infty$,  $-1/p'<s<-1/q'$.
A duality argument, see \cite{su},
 \S 2.3, shows that assertions  (i), (ii) in the theorem are  equivalent. It is
sufficient to prove the result for $N$ large.

As stated above
the theorem  was proved in \cite{su}
by a probabilistic argument which does not identify the specific projection for which the lower bound holds. We now 
give an explicit and deterministic definition of 
such projections.

Let $R$ be a large positive integer to be chosen later. Let $N\gg R$. Given a set of Haar frequencies $A\subset \{2^j: j\ge 1\}$ we
choose a $A_N\subset A$ such that 
\begin{subequations}\label{defproj}
\Be\label{Rcard}
2^{N-1}R^{-1}\le \#A_N \le 2^N,\Ee
and such that $\log_2A_N$ is $R$-separated; i.e., we have the property that
\Be \label{R-separate}
2^n\in A_N,\,\,2^{\tilde n}\in A_N \,\implies 
|n-\tilde n|\ge R \text{ if } n\neq \tilde n.\Ee
Let $E=E(N,R)$ be the collection of Haar functions $h_{j,\mu}$ with $2^j\in A_N$ and $0\le \mu\le 2^j-1$, and let $P_E$
be the orthogonal projection 
to the span of $E$, defined initially on $L^2$
\Be \label{PEdef} 
P_E f= 
 \sum_{2^j\in A_N} \sum_{\mu=0}^{2^j-1} {2^j}\inn {f}{h_{j,\mu} }h_{j,\mu}.
\Ee
\end{subequations}
We have the following main result. 
\begin{theorem}  \label{constrthm}
There is $R=R(p,q,s)>1$ and $N_0=N_0(p,q,s) $ so that for all $N\ge N_0\gg R$  the following lower bounds hold
for the projection operators $P_E$ with $E=E(N,R)$ as defined in 
\eqref{defproj}.

(i) For $1<p<q<\infty$, $1/q<s<1/p$
$$\|P_E\|_{F^s_{p,q}\to F^s_{p,q}} \ge c(p,q,s) 2^{N(s-1/q)}\,.$$

(ii)  Let $1<q<p<\infty$  and $-1+1/p<s<-1+1/q$ then
$$\|P_E\|_{F^s_{p,q}\to F^s_{p,q}} 
\ge c(p,q,s) 2^{N(\frac 1q -s-1)}\,.$$
\end{theorem}

For the proof of (ii)  we shall construct {deterministically} test functions $f_N\in F^s_{p,q}$ for which 
\Be\label{detfunct}
 \|P_E f_N\|_{F^s_{p,q}} \gc \|P_E\|_{F^s_{p,q}\to F^s_{p,q}}  \|f_N\|_{F^s_{p,q}}.\Ee
Here $q<p$ and $-1/p'<s<-1/q'$.

The paper is organized as follows. In \S2 we will recall characterizations of  the spaces $F^s_{p,q}$ in terms of local
means which are convenient to work with.
In \S3 and \S4 we
give the   construction of  a family of test functions satisfying 
\eqref{detfunct}.
\S5 and \S6 contain the core of the proof. Finally, in \S 7 we
state some open problems.

\section{Preliminaries}

Let $\psi_0, \psi \in \cS(\rr)$ such that $|\hat{\psi}_0(\xi)|>0$ on $(-\varepsilon,\varepsilon)$ and
$|\hat{\psi}(\xi)|>0$ on $\{\xi \in \rr:\eps/4<|\xi|<\eps\}$ for some fixed $\eps>0$. We further assume vanishing
moments of $\psi$ up to order $M_1$ of $\psi$; i.e.,
$$
    \int \psi(x)x^n\,dx = 0\quad\mbox{for}\quad n=0,1,...,M_1\,.
$$
As usual we define $\psi_k:=2^k\psi(2^k\cdot)$.

\begin{definition}\label{TLspaces} Let $0<p<\infty$, $0<q\leq \infty$ and $s\in \rr$. Let further $\psi_0,\psi \in
\cS(\rr)$ as above with $M_1+1>s$. The Triebel-Lizorkin space $F^s_{p,q}(\rr)$ is the collection of all tempered
distributions $f\in \cS'(\rr)$ such that 
$$
    \|f\|_{F^s_{p,q}}:=\Big\|\Big(\sum\limits_{k=0}^{\infty}2^{ksq}|\psi_k * f(\cdot)|^q\Big)^{1/q}\Big\|_p
$$
is finite.
\end{definition}

The definition of the spaces $F^s_{p,q}(\rr)$, cf.\ \cite{triebel2}, is usually given in terms of a compactly supported
(on the Fourier side) smooth dyadic decomposition of unity. Based on vector-valued singular integral theory \cite{BCP}
it can be shown that the characterization given in Definition \ref{TLspaces} is equivalent, see also \cite[\S
2.4.6]{triebel2} and \cite{Ry99}.
The above characterization allows for choosing $\psi_0, \psi$ compactly supported, which is the reason
for the term ``local means'' which in view of the  localization properties of the Haar functions are 
useful for the purpose of this paper.

\section{Families of test functions}\label{outline}
\noindent Let $A\subset \{2^j:j\ge
1\}$ be given and choose $A_N$ such that 
\Be
2^{N-1}R^{-1}\le \#A_N \le 2^N\,. \Ee
We set $\fA^N=\log_2A_N$, i.e., 
$\fA^N=\{j:2^j\in A_N\}$.
Also let, for large $N$,
\begin{align*}
\fL^N&=\{l:l-N\in \fA^N\},
\\
\fS^N&=\{(l,\nu):l\in \fL^N,\,0\le \nu\le 2^{l}-1,
\,\,\nu \in 2^{N}\bbZ+2^{N-1}\},
\\
\fS^N_l&=\{\nu: (l,\nu)\in \fS^N\}.
\end{align*}

\noindent Let $\eta$ be an odd $C^\infty$ function supported in $(-2^{-4},2^{-4})$. Furthermore it is assumed that
$\eta$ has vanishing moments up to order $M_0$, i.e., $\int \eta(x) x^n\,dx=0$ for $n=0,1,\dots, M_0$ (with $M_0$ some
large constant), and that
\Be \label{etanondeg}
2\int_{0}^{1/2} \eta(x) dx = 
\int_{0}^{1/2} \eta(x) dx -  \int_{-1/2}^0 \eta(x) dx  
\ge 1\,.
\Ee
We further define 
\Be\label{etalnudef}
\eta_{l,\nu}(x) = 
\eta(2^l(x-x_{l,\nu}))\,,\quad  \text{where}\quad  x_{l,\nu}=2^{-l}\nu\,.
\Ee Then, clearly, $\eta_{l,\nu}$ is supported in $[x_{l,\nu}-2^{-l-4}, x_{l,\nu} +2^{-l-4}]$.
Crucial for the subsequent analysis is the fact that for  $\nu, \nu' \in \fS^N_l$ with $\nu \neq \nu'$  the distance
of the supports of $\eta_{l,\nu}$ and
$\eta_{l,\nu'}$ is at least $2^{N-l}- 2^{-l-3}$.

Let us define the family of test functions $f_{N}$ by 
\Be \label{testfct} 
f_{N}(x) := \sum\limits_{l \in \fL^N}2^{-ls}\sum_{\nu \in \fS^N_l} \eta_{l,\nu}(x).\Ee
The following $F^s_{p,q}$-norm bound for  the $f_N$ follows from \cite[Prop.\ 4.1]{su}, which is based
on a result in \cite{cs-lms}.

\begin{proposition}\label{testfctbound} Let $1\leq q\leq p<\infty$, $s\in \rr$ and $s>-M_0$. Then 
$$
    \|f_N\|_{F^s_{p,q}} \lesssim_{p,q,s} (1+2^{-N}\#(\fL^N))^{1/q} \lesssim 1\,.
$$
\end{proposition}

\section{Lower bounds for Haar projections}

\noindent Let $P_E$ be the (family of) projections defined in \eqref{PEdef}. By Proposition \ref{testfctbound} it
suffices to show 
\Be\label{lowerboundPEfS}
\|P_E f_N\|_{F^s_{p,q}} \ge c(p,q,s, R) 2^{N(\frac 1q -s-1)}
\Ee
in the case $1<q<p$, $-1/p'<s<-1/q'$. What remains follows by duality, cf.\ \cite[\S 2.3]{su}.

Let  $\psi$ be a non-vanishing $C^\infty$-function supported on $(-2^{-4},2^{-4})$ in the sense of Definition
\ref{TLspaces} with $M_1$ large enough. Setting $\psi_k=2^k\psi(2^k\cdot)$ we have the inequality (according to
Definition \ref{TLspaces}),
$$\Big\|\Big(\sum_{k=1}^\infty 2^{ksq}|\psi_k*g|^q\Big)^{1/q}\Big\|_q\lc\|g\|_{F^s_{p,q}}\,.
$$
Hence, it now suffices to show (for large $N$)
\Be\label{locallowerbdpq}
\Big\|\Big(\sum_{k\in \fA^N}2^{ksq}|\psi_k *P_E f_N|^q\Big)^{1/q}\Big\|_p
\gc 2^{N(\frac 1q-s-1)}\,.\Ee

Now $\psi_k*P_Ef_S$ is supported in $[-1,2]$ and, by H\"older's inequality and  $p\ge q$,  it is enough to verify 
\eqref{locallowerbdpq} for $p=q$, i.e., we have to show
\begin{multline}
 \label{locallowerbdq}
\Big\|\Big(\sum_{k\in \fA^N}2^{ksq}\Big|
\sum_{(l,\nu)\in \fS^N}
2^{-ls} \sum_{j\in \fA^N}\sum_{\mu=0}^{2^j-1} \inn{\eta_{l,\nu}}{h_{j,\mu}} h_{j,\mu}*\psi_k
\Big|^q\Big)^{1/q}\Big\|_q\\
 \gc 2^{N(\frac 1q-s-1)}.\end{multline}
Let us define \Be\label{GjNkl} \Gf(x) \,=\,
\sum_{\mu=0}^{2^j-1} \sum_{\nu\in \fS^N_l}
 2^j \inn{\eta_{l,\nu}}{h_{j,\mu}} h_{j,\mu}*\psi_k. \Ee
Recall that for $k\in \fA^N$ we have $k+N\in \fL^N$. We shall show two inequalities. In what follows we
always have $1<q<p<\infty$, $-1/p'\leq s<-1/q'$.
\begin{proposition} \label{lower-boundprop} 
There is  $c_1>0$ such that 
\Be\label{lower-bound}
\Big(\sum_{k\in \fA^N} 2^{ksq}\big\| 2^{-(k+N)s} G^{k,N}_{k, k+N}
 \big\|_q^q\Big)^{1/q}
\ge c_1 R^{-1/q}2^{N(\frac 1q-1-s)}\,.
\Ee
\end{proposition} 

\begin{proposition}  \label{upper-boundprop} 
There is $\eps=\eps(s,q)>0$ such that 
\Be\label{ubound}\Big(\sum_{k\in \fA^N} 2^{ksq}\Big\|
 \sum_{\substack{(j,l)\in \fA^N\times\fL^N\\(j,l)\neq(k, k+N)}}2^{-ls} \Gf
\Big\|_q^q\Big)^{1/q}
\le C_2 (1+2^{-R\eps} 2^{N(\frac 1q-1-s)})\,.
\Ee
\end{proposition}
\noindent If we choose $R$ large enough (depending on $p,q,s$) then the two propositions imply 
Theorem \ref{constrthm}.

\section{Proof of Proposition \ref{lower-boundprop}    }
\label{lowbdsect} 
\noindent Let $\Psi(x)= \int_{-\infty}^x \psi(t) dt$,  supported also in $(-2^{-4}, 2^{-4})$.
Note that
$$\psi*h_{0,0} (x) = \Psi(x)+\Psi(x-1)-2\Psi(x-\tfrac 12)$$
and hence  $\psi*h_{0,0} (x) = -2\Psi(x-\tfrac 12)$ if $x\in [1/4,3/4]$.
Thus there is $c>0$ and  an interval $J\subset [1/4,3/4]$ so that
$$|\psi* h_{0,0}(x)|\ge  c  \text{ for } x\in J\,.$$
For $k=0,1,2,\dots$ and $\mu\in \bbZ$ let $J_{k,\mu}=2^{-k}\mu +2^{-k}J$, 
a subinterval of the middle half of $I_{k,\mu}$ of length $\gc 2^{-k}$. We then get the estimate \Be
\label{lowerboundforconv}
|\psi_k*h_{k,\mu}(x)|\ge c_0 \text{ for } x\in J_{k,\mu}.
\Ee
The left-hand side of   \eqref{lower-bound}
 is
\begin{align*}
2^{-Ns}\Big(\sum_{k\in \fA^N} \Big\|
\sum_{\mu=0}^{2^k-1}\sum_{\nu\in \fS^N_{k+N}}
2^k\inn{h_{k,\mu}}{\eta_{k+N,\nu}} h_{k,\mu}*\psi_k\Big\|_q^q
\Big)^{1/q}\,.
\end{align*}
Now $h_{k,\mu} $ is supported in $I_{k,\mu}=[2^{-k}\mu, 2^{-k}(\mu+1)]$.
If $\nu\in \fS^N_{k+N}$ then $\nu=2^N m + 2^{N-1}$ for some  integer $m$ and in this case $\eta_{k+N,\nu}$ is
supported in $[2^{-k}m+2^{-k-1}-2^{-k-N-4},
2^{-k}m+2^{-k-1}+2^{-k-N-4}]$. For fixed $\mu$ this interval intersects
$I_{k,\mu}$ only if $m=\mu$, and thus for the scalar products 
$\inn{h_{k,\mu}}{\eta_{k+N,\nu}}$ (with $\nu\in \fS^N_{k+N}$) 
 we only get a contribution for  $\nu=\nu_N(\mu):= 2^{N}\mu+2^{N-1}$. We calculate
\begin{align}
\langle h_{k,\mu} \eta_{k+N,\nu_N(\mu)}\rangle\notag
 &=\int h_{0,0}(2^k x-\mu)
 \eta(2^{k+N}( x- 2^{-k}\mu-2^{-k-1})) dx
\notag
\\
&=2^{-k}\int\eta(2^N(y-1/2))h_{0,0}(y)\,dy\notag\\
&=2^{-k}\int_{-1/2}^0 \eta(2^Ny)-2^{-k}\int_0^{1/2}\eta(2^Ny)\,dy\notag\\
&=-2^{-N-k+1}\int_0^{1/2}\eta(u)\,du\notag\,,
\end{align}
where  we have used that $\supp (\eta(2^N\cdot))$ is 
contained in $(-2^{-N-4}, 2^{-N-4})$.
By \eqref{etanondeg} we get
$$ \big|\inn{h_{k,\mu}} {\eta_{k+N,\nu_N(\mu)}}\big|
\ge 2^{-k-N}.$$
Recall that $J_{k,\mu'}$ is contained in the middle half of $I_{k,\mu'}$. Now
 $$\supp(\psi_k*h_{k,\mu}) \subset [2^{-k}\mu-2^{-k-4}, 2^{-k}(\mu+1)+2^{-k-4}]$$
and thus, given $\mu,\mu'$, the support of $\psi_k*h_{k,\mu}$
can intersect $J_{k,\mu'}$ only if $\mu=\mu'$. Hence, if we set
$\Om_k=\cup_{\mu=1}^{2^{k-1}} J_{k,\mu}$ we have, using also \eqref{lowerboundforconv},
\begin{align*}
\big\|G^{k,N}_{k,k+N}
\big\|_q^q &\ge  \int_{\Om_k} |G^{k,N}_{k,k+N}(x)|^q dx
\\
&\ge\sum_{\mu=1}^{2^{k-1}} \big| 2^k \inn{h_{k,\mu}} {\eta_{k+N,\nu_N(\mu)}}
\big|^q\int_{J_{k,\mu}} |\psi_k*h_{k,\mu}(x)|^q dx
\\
&\ge c 2^{-Nq}\,,
\end{align*}
where $c>0$ does neither depend on $R$ nor $N$.
Since $\card(\fA^N)\gc 2^{N-1}/R$ we obtain the lower bound
\eqref{lower-bound}  after summing in $k$.

\medskip

\section{Proof of Proposition \ref{upper-boundprop}    }\label{upperbdsect}
We first collect several standard and elementary facts 
about the Haar coefficients.

\begin{lemma} \label{haarcoeff}
(i) If  $\supp (\eta_{l,\nu})$ is contained
 either in $I^+_{j,\mu}$, or in
$I^-_{j,\mu}$, or in
$I^\complement_{j,\mu}$ then 
$\inn{\eta_{l,\nu}}{h_{j,\mu}} =0\,.$

(ii)  $$
|\inn{\eta_{l,\nu}}{h_{j,\mu}}|
 \lc \begin{cases} 2^{-l}& \text{ if } l\ge j,
\\2^{l-2j} &\text{ if }  l\le j.\end{cases}
 $$
\end{lemma}

\begin{lemma} \label{canc-const}
(i) Suppose that   $k\ge j$. If the distance of $x$ to the three points 
$ 2^{-j}\mu$, $2^{-j}(\mu+\tfrac12)$, $2^{-j}(\mu+1)$ is at least $2^{-k}$ then 
$h_{j,\mu}*\psi_k(x)= 0$.

(ii)
For $k\ge j$ we have  $\|h_{j,\mu}*\psi_k\|_q\lc 2^{-k/q}$.
\end{lemma}

\begin{lemma} \label{psicanc}
Let  $k\le j$. Let  $y_{j,\mu}:= 2^{-j}(\mu+\tfrac 12)$, the midpoint of the interval $I_{k,\mu}$. Then  the support of
$h_{j,\mu}*\psi_k$ is contained in $[y_{j,\mu} - 2^{-k},y_{j,\mu} +2^{-k}].$ Also,
$$\|h_{j,\mu}*\psi_k\|_\infty \lc 2^{2k-2j}.$$
\end{lemma}
The proofs of 
Lemmata \ref{haarcoeff}, \ref{canc-const} and \ref{psicanc}
are straightforward and can be looked up e.g. in \cite{su}.


We have the following estimates when  $j\le k$.
\begin{lemma}\label{jklemma} Let $l\geq N$. For $1\le q\le \infty$,
\begin{subequations}
\begin{alignat}{3}
\label{jk-1}&\|\Gf\|_q
\lc 2^{j-l} 2^{(j-k)/q},&  &k\ge j, \quad
l\ge j+N,
\\
\label{jk-2}
&\|\Gf\|_q\lc 2^{j-l} 2^{(l-N-k)/q},&
&k\ge j, \quad
 j\le l\le j+N,
\\ \label{jk-3}
&\|\Gf\|_q\lc 2^{l-j}2^{(l-k-N)/q},
\qquad&  &k\ge j, \quad l\le j.
\end{alignat}
\end{subequations}\end{lemma}
\begin{proof}
Let $l\ge j+N$.
 By Lemma \ref{canc-const}, (i), the function $\Gf$ is supported on the union of $O(2^j)$ intervals of length $2^{-k}$,
i.e. on a set of measure $O(2^{j-k})$.
By Lemma \ref{haarcoeff} we have, for fixed $\mu$, that
$\inn{\eta_{l,\nu}}{h_{j,\mu}}\neq 0$
only for a finite number of indices $\nu$, and we always have
$2^j|\inn{\eta_{l,\nu}}{h_{j,\mu}}|\lc 2^{j-l}$. Thus \eqref{jk-1} follows.

Now let $j\le l\le j+N$. Since the sets $\supp (\eta_{l,\nu})$ with $\nu \in \fS^N_l$ are
$2^{N-2-l}$ separated, and $2^{N-l}\ge 2^{-j}$, we see from Lemma \ref{canc-const} that
$\Gf$ is supported on the union of $O(2^{l-N})$ intervals of length $2^{-k}$,
i.e. on a set of measure $O(2^{l-k-N})$.
As in the previous case $2^j|\inn{\eta_{l,\nu}}{h_{j,\mu}}|\lc 2^{j-l}$,
 and  \eqref{jk-2} follows.

Let $l\le j$.
As in the previous case $\Gf$ is supported on a set of measure $O(2^{l-k-N})$.
By Lemma \ref{haarcoeff}, (ii),
we have now
$2^j|\inn{\eta_{l,\nu}}{h_{j,\mu}}|\lc 2^{l-j}$, and
 \eqref{jk-3} follows.
\end{proof}

For $k\le j$ we have
\begin{lemma}\label{kjlemma} Let $l\ge N$.  For $1\le q\le \infty$,
\begin{subequations}
\begin{alignat}{3}
\label{kj-1}
&\|\Gf\|_q
\lc 2^{k-l},&
&\quad k\le j\le l-N,
\\
\label{kj-2}
&\|\Gf\|_q
\lc 2^{k-j-N},&
&\quad k\le l-N\le j\le l,
\\ \label{kj-3}
&\|\Gf\|_q
\lc  2^{2k-j-l} 2^{(l-k-N)/q},&
& \quad l-N\le k\le j\le l,
\\ \label{kj-4}
&\|\Gf\|_q
\lc 2^{2k-2j} 2^{(l-k-N)/q},&
& \quad l-N\le k\le l\le j,
\\ \label{kj-5}
&\|\Gf\|_q
\lc 2^{l+k-2j-N},& &\quad k\le l-N\le l\le j,
\\ \label{kj-6}
&\|\Gf\|_q
\lc  2^{3k+l-4j} 2^{-N/q},&
& \quad l\le k\le j.
\end{alignat}
\end{subequations}\end{lemma}
\begin{proof}
Let, for $\rho\in \bbZ$,  $I_{k,\rho}^*=\cup_{i=-1,0,1}I_{k,\rho+i}$ the triple interval. Then for
$j\ge k$ the function $h_{j,\mu}*\psi_k$ is supported in at most
five of the intervals $I_{k,\rho}^*$.

For the case \eqref{kj-1} we have  $2^{-l+N}\le 2^{-j}\le 2^{-k}$. By Lemma
\ref{haarcoeff} we have, for each $\mu$,
$\sum_\nu 2^j|\inn{\eta_{l,\nu}}{h_{j,\mu}}
|\lc 2^{j-l}$. By Lemma \ref{psicanc}
$|h_{j,\mu}*\psi_k(x)|
\lc 2^{2k-2j}$ and for fixed $x$ there are at most $O(2^{j-k})$ terms with
$h_{j,\mu}*\psi_k(x)\neq 0$.  Hence $|\Gf(x)|\lc 2^{k-l}$ and \eqref{kj-1} follows.

Now consider the case \eqref{kj-2}, $2^{-l}\le 2^{-j}\le 2^{-l+N}\le 2^{-k}$.
Let $M_k(x)$ be the number of indices $\mu$ for which there exists
a $\nu$ with
$\inn{\eta_{l,\nu}}{h_{j,\mu}}\neq 0$
 and for which $h_{j,\mu}*\psi_k(x)\neq 0$.
Since the supports of the $\eta_{l,\nu}$ are $2^{N-2-l}$ separated, and $2^{N-l}\ge 2^{-j}$
we have $M_k(x)\lc 2^{l-N-k}$.
The upper bounds for $\sum_\nu 2^j|\inn{\eta_{l,\nu}}{h_{j,\mu}}|$ and for
$|h_{j,\mu}*\psi_k(x)|$ are as in the previous case.
 Hence
$|\Gf(x)|\lc 2^{j-l} 2^{2k-2j} 2^{l-N-k}=2^{k-j-N}$ and \eqref{kj-2} follows.

Next consider the case \eqref{kj-3}, $2^{-l}\le 2^{-j}\le 2^{-k}\le  2^{-l+N}$.
As in the previous case, $|h_{j,\mu}*\psi_k(x)|\lc 2^{2k-2j}$, and
$2^j|\inn{\eta_{l,\nu}}{h_{j,\mu}}
|\lc 2^{j-l}$.
Also since the supports of the $\eta_{l,\nu}$ are $2^{-l+N-2} $-separated and $2^{-l+N}\ge 2^{-k}\ge 2^{-j}\ge 2^{-l}$  
there are,
for every $x$  only $O(1)$
indices $\nu$, and $O(1)$ indices $\mu$ so that
$\inn{\eta_{l,\nu}}{h_{j,\mu}}\neq 0$ and
$h_{j,\mu}*\psi_k(x)\neq 0$. Hence
$\|Gf\|_\infty\lc 2^{2k-2j}2^{j-l}=2^{2k-j-l}$.
Finally, again, because $2^{-l+N}\ge 2^{-k}$  the support  of $\Gf$ is contained in a union of $O(2^{l-N})$ intervals of
length $O(2^{-k})$ and thus in a set of measure $O(2^{l-N-k})$. Now
\eqref{kj-3} follows.

Consider  the case  \eqref{kj-4}, $2^{-j}\le 2^{-l}\le 2^{-k}\le  2^{-l+N}$.
Since $2^{-l+N}\ge 2^{-k}$ there are,  for any $x$, only  $O(1)$ indices $\nu$ such there exists a $\mu$ with
$\inn{\eta_{l,\nu}}{h_{j,\mu}}\neq 0$ and
$h_{j,\mu}*\psi_k(x)\neq 0$; moreover the set of $x$ for which this can happen is
a union of $O(2^{l-N})$ intervals of length $O(2^{-k})$
and thus  of measure $O(2^{l-N-k})$. By Lemma \ref{psicanc},
$\|h_{j,\mu}*\Psi_k\|_\infty \lc 2^{2k-2j}$, and by Lemma \ref{haarcoeff} we have, for fixed $\nu$,
$\sum_\mu2^j|\inn{\eta_{l,\nu}}{h_{j,\mu}}| \lc 2^{j-l} 2^j 2^{l-2j} \lc 1$ and thus
$\|\Gf\|_\infty \lc 2^{2k-2j}$. Together with the support property of $\Gf$ this
shows \eqref{kj-4}.

Next consider the case \eqref{kj-5},
$2^{-j}\le 2^{-l}\le 2^{N-l}\le 2^{-k} $.
For each $x$ we have
$\psi_k*h_{j,\mu}(x)\neq 0$ only for those $\mu$ with
$|2^{-j}\mu-x| \le 2\cdot 2^{-k}$.
We can have $\inn{\eta_{l,\nu}}{h_{j,\mu}}\neq 0$
 for some of such $\mu$ only when
$|2^{-l}\nu-x|\le 2\cdot2^{-k}$ and
 because of the $2^{-l+N-2}$-separateness of the sets $\supp(\eta_{l,\nu})$
there are at most $2^{l-N-k} $ indices $\nu$ with this property. For each such $\nu$
there are at most $O(2^{j-l})$ indices $\mu$ such that
$\inn{\eta_{l,\nu}}{h_{j,\mu}}\neq 0$.
We use the bounds
$2^j\inn{\eta_{l,\nu}}{h_{j,\mu}}=O(2^{l-j})$ and
$\psi_k*h_{j,\mu}(x)=O(2^{2k-2j})$ to see that
$|\Gf(x)|\lc 2^{l-N-k} 2^{j-l} 2^{l-j} 2^{2k-2j}$;
hence $\|\Gf\|_\infty \lc 2^{l+k-2j-N}$ which gives
\eqref{kj-5}.

Finally, for \eqref{kj-6}, $2^{-j}\le 2^{-k}\le 2^{-l}$, we use $l\ge N$.  Then
by the separation of the sets $\supp(\eta_{l,\nu})$  we see that
$\Gf$ is supported  on the union of $O(2^{l-N})$ intervals
$I_\nu$
of length $O(2^{-l})$  (containing $2^{-l}\nu$ with $\nu\in \fS^N_l$). Thus
$\Gf$ is supported on a set of measure $2^{-N}$.
For $x\in I_\nu$ there are at most $O(2^{k-j})$ indices $\mu$ with $\psi_k*h_{j,\mu}(x)\neq 0$. 
For any such $\mu$ we have again $2^j\inn{\eta_{l,\nu}}{h_{j,\mu}}=O(2^{l-j})$ and
$\psi_k*h_{j,\mu}(x)=O(2^{2k-2j})$. Thus we have the bound
$|\Gf(x)|\lc 2^{k-j}  2^{l-j} 2^{2k-2j}$; hence
$\|\Gf\|_\infty \lc 2^{3k+l-4j}$ and by the estimate for the support of $\Gf$ we
obtain \eqref{kj-6}.
\end{proof}

Now let $\cP$ be the set of pairs $(m,n)\in \bbZ\times \bbZ$ such that at least one of the  inequalities
$|m|\ge R$, $|n|\ge R$ is satisfied.
Change variables to  write $j=k+m$, and $l=k+N+n$. We estimate the
left hand side
of \eqref{ubound} as
\Be \label{GleT}\Big(\sum_{k\in \fA^N} \Big\|
 \sum_{\substack{(m,n):\\(k+m,k+N+n)\in \fA^N\times\fL^N\\(m,n)\neq(0,0)}}
2^{-s(n+N)} G^{k+m,N}_{k,k+N+n}
\Big\|_q^q\Big)^{1/q} \lc \sum_{(m,n)\in \cP} \cV_{m,n}
\Ee
where $$\cV_{m,n}
 = 2^{-s(n+N)} \Big(\sum_{\substack{k\in \fA^N:\\(k+m,k+N+n)\in \fA^N\times\fL^N}}
\| G^{k+m,N}_{k,k+N+n}
\|_q^q\Big)^{1/q}.$$

We may rewrite the inequalities in Lemma \ref{jklemma} and Lemma \ref{kjlemma} and get estimates in terms of $m,n$.
Using $\#\fA^N=O(2^N)$ this leads to
the following inequalities for $m\le 0$.
\begin{subequations}
\begin{alignat}{3}
\label{mn-1}&\cV_{m,n}
\lc 2^{N(\frac 1q-1-s)}
 2^{-n(1+s)+m(1+\frac 1q)},\quad&  &m\le 0, \quad
n\ge m,&
\\
\label{mn-2}
&\cV_{m,n}
\lc 2^{N(\frac 1q-s-1)}2^{n(\frac 1q -1-s)+m} , &
&m\le 0, \quad
 m-N\le n\le m,&
\\ \label{mn-3}
&\cV_{m,n}\lc 2^{(N+n)(\frac 1q-s)}2^{N+n-m},
\qquad&  &m\le 0, \quad n\le m-N.&
\end{alignat}
\end{subequations}
For $m\ge 0$ we get from Lemma \ref{kjlemma}
\begin{subequations}
\begin{alignat}{3}
\label{nm-1}
&\cV_{m,n}
\lc 2^{N(\frac 1q-s-1)}
2^{-n(1+s)},&
&\quad 0\le m\le n,&
\\
\label{nm-2}
&\cV_{m,n}
\lc
2^{N(\frac 1q-s-1)}2^{-m-sn},&
&\quad 0\le n\le m\le n+N,&
\\ \label{nm-3}
&\cV_{m,n}
\lc  2^{N(\frac 1q-s-1)}2^{n(\frac 1q-1-s)-m},  &
& \quad n\le 0\le m\le n+N,&
\\ \label{nm-4}
&\cV_{m,n}
\lc 2^{N(\frac 1q-s)} 2^{-2m+n(\frac 1q-s)},&
& \quad n\le 0\le n+N\le m,&
\\ \label{nm-5}
&\cV_{m,n}
\lc 2^{N(\frac 1q-s)} 2^{-2m}2^{n(1-s)},&
& \quad 0\le n\le m-N,&
\\ \label{nm-6}
&\cV_{m,n}
\lc 2^{(N+n)(1-s)-4m} ,&
& \quad n+N\le 0\le m.&
\end{alignat}
\end{subequations}
Now use
 the assumption  $-1<s<\frac 1q-1$, and conclude by summing in $m$, $n$ for the various parts.
First consider the case $m\le 0$. By \eqref{mn-1},
\begin{subequations}\label{sumTnegm}
\begin{align}
\notag
\sum_{\substack {(m,n)\in \cP\\ m\le 0,\, n\ge m}} \cV_{m,n} &\lc
2^{N(\frac 1q-1-s)}\Big(
 \sum_{m\le -R}  \sum_{n\ge m} +\sum_{-R\le m\le 0}  \sum_{n\ge R}\Big)
2^{m(1+\frac 1q)}
2^{-n(1+s)}\\
\label{Tmn-1}
&\lc\big( 2^{-R(\frac 1q-s)}+2^{-R(1+s)}\big) \, 2^{N(\frac 1q-1-s)}\,.
\end{align}
By \eqref{mn-2},
\begin{align}
\notag
\sum_{\substack {(m,n)\in \cP\\ m\le 0,\, m-N\le n\le m}} \cV_{m,n} &\lc
\Big(
 \sum_{m\le -R}  \sum_{n\le m} +\sum_{-R\le m\le 0}  \sum_{n\le -R}\Big)
2^{m+(n+N)(\frac 1q-s-1)}\\ \label{Tmn-2}
&\lc 2^{-R(\frac 1q-s-1)} \, 2^{N(\frac 1q-1-s)}\,.
\end{align}
By \eqref{mn-3},
\Be \label{Tmn-3}
\sum_{\substack {(m,n)\in \cP\\ n+N\le m\le 0}} \cV_{m,n} \lc
\sum_{m\le 0} \sum_{n\le m-N}2^{(n+N)(\frac 1q-s+1)-m}\lc_{q,s} 1\,.
\Ee
\end{subequations}

We now turn to the terms with $m\ge 0$. Observe that the conditions
$(m,n)\in \cP$, $ 0\le m\le n$ imply $ n\ge R$, and by \eqref{nm-1}
we get
\begin{subequations}\label{sumTposm}
\begin{align}
\notag
\sum_{\substack {(m,n)\in \cP\\ 0\le m\le n-N}} \cV_{m,n} &\lc
2^{N(\frac 1q-1-s)}\sum_{n\ge R} 2^{-n(1+s)}n
\\ \label{Tnm-1}
&\lc R2^{-R(1+s)} 2^{N(\frac 1q-1-s)}\,.
\end{align}
By \eqref{nm-2},
\begin{align}
\notag
\sum_{\substack {(m,n)\in \cP\\ 0\le n\le m\le n+N}} \cV_{m,n} &\lc
2^{N(\frac 1q -1-s)}\Big( \sum_{m\ge R}\sum_{0\le n\le m}+\sum_{n\ge R}\sum_{m\ge n}\Big)2^{-sn}2^{-m}
\\ \label{Tnm-2}&\lc 2^{-R(1+s)} 2^{N(\frac 1q-1-s)}\,.
\end{align}
By \eqref{nm-3},
\begin{align}
\notag
\sum_{\substack {(m,n)\in \cP\\ n\le 0\le m\le n+N}} \cV_{m,n} &\lc
\Big(\sum_{-N\le n\le 0}\sum_{m\ge R} + \sum_{-N\le n\le -R}\sum_{0\le m\le N+n}\Big)
 2^{(N+n)(\frac 1q-1-s)} 2^{-m}
\\\label{Tnm-3}
&\lc \big(2^{-R}+ 2^{-R(\frac 1q-1-s)}\big)  2^{N(\frac 1q-1-s)}\,.
\end{align}
By \eqref{nm-4},
\begin{align}\notag
&\sum_{\substack {(m,n)\in \cP\\ n\le 0\le  n+N\le m}} \cV_{m,n} \lc
\sum_{-N\le n\le 0}\sum_{m\ge n+N} 2^{-2m+(n+N)(\frac 1q-s)}
\\ \label{Tnm-4}
&\lc \sum_{-N\le n\le 0} 2^{(n+N)(\frac 1q-s-2)} \lc 1\,.
\end{align}
By \eqref{nm-5},
\begin{align}
\notag
\sum_{\substack {(m,n)\in \cP\\ 0\le n\le m-N}} \cV_{m,n} &\lc
2^{N(\frac 1q-s)} \sum_{n\ge 0} 2^{n(1-s)} \sum_{m\ge n+N}2^{-2m}
\\ \label{Tnm-5}
&\lc 1\,.
\end{align}
By \eqref{nm-6},
\Be \label{Tnm-6}
\sum_{\substack {(m,n)\in \cP\\ n+N\le 0\le m}} \cV_{m,n} \lc
\sum_{n\le -N} 2^{(n+N)(1-s)}\sum_{m\ge 0} 2^{-4m} \lc 1\,.
\Ee

\end{subequations}
We combine \eqref{GleT} with the various estimates
in \eqref{sumTnegm} and \eqref{sumTposm} to obtain \eqref{ubound} with a positive
$\eps< \min \{s+1, \frac 1q-s-1\}$.

\section{Concluding remarks}

\subsection{\it Endpoint cases.}
For the endpoint cases $p<q$, $s=1/q$, and $q<p$, $s= -1/q'$  it has been shown 
in \cite{su} that for any $A\subset\{2^j: j\in \bbN\}$ with  $\#A\approx 2^N$ 
there is a set $E$ of Haar functions 
supported in $[0,1]$ such that $\HF(E)\subset A$ and such that
\begin{align*}
&\|P_E\|_{F^{1/q}_{p,q} \to F^{1/q}_{p,q} }
\gc N^{1/q}, \quad 1<p<q<\infty
\\
&\|P_E\|_{F^{-1/q'}_{p,q} \to F^{-1/q'}_{p,q} }
\gc N^{1-1/q}, \quad 1<q<p<\infty
\end{align*}
If $A$ is $N$-separated then these bounds are sharp, as they are matched with corresponding upper bounds.
In all cases the upper bounds are $O(N)$ and in some cases the lower bounds may be $\approx N$.  The proofs of  these results in \cite{su} rely on probabilistic arguments.  A combination with the ideas  in this paper
(using in particular the  $R$-separation in the frequency sets in 
\eqref{R-separate})  also yields lower bounds for explicit examples of  projections. The details are somewhat lengthy (cf.   \S\ref{upperbdsect}) and we shall not pursue this here.
\subsection{\it Some open problems}
\subsubsection{Test functions for the case $p<q$}
Our proofs reduce the case $p<q$ to the case $q<p$ by duality. It would be interesting to identify suitable test functions in the case $p<q$ and get a proof which establishes directly the lower bounds in the range  $1/q\le s<1/p$.
\subsubsection{\it Multipliers for Haar expansions}
Let $m\in \ell^\infty(\bbN\times \bbZ)$. Consider the operator
defined on $L^2$ by
$$ T_mf = \sum_{j,\mu} m(j,\mu) 2^{j}\inn{f}{h_{j,\mu} }h_{j,\mu}\,.$$
 When 
$\max\{-1/p',-1/q'\}<s<\min \{1/p, 1/q\}$ the operator $T_m$ is bounded on $F^s_{p,q}$  with operator norm $\lc \|m\|_\infty$
since the Haar system is an unconditional basis in this case. What is the right condition 
for boundedness on
$F^s_{p,q}$  when  
either $p<q$, $1/q\le s<1/p$, or $q<p$, $-1/p'<s\le -1/q'$? 
\subsubsection{\it Quasi-greedy bases} Compared to unconditionality there is a weaker property of a basis in a Banach
space, called ``quasi-greedy'', see \cite[\S 1.4]{Tbook}, which is highly relevant for non-linear
approximation. It is known that ``unconditionality'' implies ``quasi-greedy'' but not the other way around, see
\cite[\S 1.1]{Tbook}. Hence it is  a natural question (asked by V. Temlyakov)
whether the Haar basis $\mathscr{H}$ is quasi-greedy in 
$F^s_{p,q}$ if $1<p<q$, $1/q\leq s< 1/p$ and $1<q<p<\infty$, $-1/p'<
s\leq -1/q'$.  Note that this is already  open in the case $q=2$, corresponding to   $L^p$ Sobolev spaces $L^s_p$.


\bigskip

{\bf Acknowledgment.} The authors would like to thank the organizers of the 2014 trimester program ``Harmonic
Analysis and Partial Differential Equations'' at the Hausdorff Research Institute for Mathematics in Bonn, where this
work has been
initiated, for providing a pleasant and fruitful research atmosphere. Both authors would like to thank Peter Oswald,
Winfried Sickel,
Vladimir N. Temlyakov and Hans Triebel for several valuable remarks.

\end{document}